\documentclass[11pt,reqno]{amsart}
\usepackage{amssymb,latexsym}
\usepackage{amsmath,color}
\usepackage{amsthm}
\usepackage{epsfig}
\usepackage{graphicx}
\usepackage{titletoc}
\usepackage{dsfont}
\usepackage{nccbbb}
\usepackage{amsmath}
\usepackage{amssymb}
\usepackage{amssymb}
\usepackage{amstext}
\usepackage{amsopn}
\usepackage{amsbsy}
\usepackage{amscd}
\usepackage{amsxtra}
\usepackage{accents}
\usepackage{bbm}
\usepackage{yfonts}
\usepackage{bbm}

\numberwithin{equation}{section}

\usepackage{xparse}
\usepackage{microtype}
\usepackage[margin=1cm]{geometry}

\usepackage{lipsum}

\ExplSyntaxOn
\box_new:N \l_hideparbox_box

\NewDocumentCommand{\hideparbox}{O{c}mm+m}
 {
  \group_begin:
  \vbox_set:Nn \l_hideparbox_box
   {
    \use:c { @parboxrestore }
    \hsize=#3\scan_stop:
    \strut#4\par
   }
  \vbadness=\c_ten_thousand 
  \vbox_set_split_to_ht:NNn \l_hideparbox_box \l_hideparbox_box { #2 }
  \parbox[#1][#2]{#3}
   {
    \vbox_unpack:N \l_hideparbox_box
   }
  \group_end:
 }
\ExplSyntaxOff

\usepackage[T1]{fontenc}
\newtheorem{theorem}{Theorem}[section]

\newtheorem{lemma}[theorem]{Lemma}

\newtheorem{remark}[theorem]{Remark}

\theoremstyle{definition}

\newcommand{\eps}{\varepsilon}

\newcommand{\dsx}{\mathsf{d}\sigma_x}

\newcommand{\dx}{\,\mathrm{d}x}

\usepackage[misc]{ifsym}
\allowdisplaybreaks

\title[A non-existence result for a nonlinear Neumann  problem]{}

\address{\rm (Chiun-Chang Lee) Institute for Computational and Modeling Science, National Tsing Hua University, Hsinchu 30013, Taiwan}
\email{chlee@mail.nd.nthu.edu.tw}

\begin{document}

\maketitle
\vspace*{-66pt}
\begin{center}
{\Large\bf  A non-existence result for a nonlinear Neumann  problem }\vspace*{8pt}\\
 Chiun-Chang Lee
\end{center}
\begin{abstract} 
{\scriptsize In this note we consider a semilinear elliptic equation in $B_R$ with the nonlinear boundary condition, where $B_R$ is a ball of radius $R$. Under certain conditions, we  establish a sufficient condition on the non-existence of solutions provided that $R$ is sufficiently large. The main argument is based on applying the asymptotic analysis to the equation with respect to $R\gg1$. 
\vspace{3pt}
}
\end{abstract}

\section{\bf Introduction} 
Let $R>0$ and $N\geq2$. We consider a semilinear elliptic equation
\begin{align}\label{equ}
- u''(r)-\frac{N-1}{r}u'(r)  +f(u(r))=0,\qquad\,r\in(0,R), 
\end{align}
with the nonlinear Neumann boundary condition
\begin{equation}\label{bdu}
    u'(0)=0,\quad\,u'(R)=g(u(R)).
\end{equation}
\eqref{equ}--\eqref{bdu} is the $N$-dimensional radial version of the nonlinear Neumann problem (cf. \cite{Lee2020,S1993}):
\begin{align}\label{ndu}
\begin{cases}
\displaystyle-\Delta{u}+f(u)=0&\quad\,\text{in}\,\,B_R,\\ 
\displaystyle\quad\frac{\partial{u}}{\partial\vec{\nu}}=g(u)&\quad\text{on}\,\,\partial{B_R},
\end{cases}
\end{align}
where $\Delta$ stands for the Laplacian operator, $B_R$ is a ball of radius $R$ centered at the origin in $\mathbb{R}^N$, and $\frac{\partial}{\partial\vec{\nu}}$ is the normal derivative with respect to the unit outward normal vector $\vec{\nu}$ to $\partial{B_R}$. We refer the reader to \cite{A1976,D2006} for the physical background of nonlinear boundary conditions and references therein.

The associated energy functional  of \eqref{ndu} is defined by
\begin{align*}
\mathcal{E}[u]=\int_{{B}_R}\left(\frac{|\nabla u|^2}{2}+F(u)\right)\dx-\int_{\partial{{B}_R}}\int_{0}^{u}g(t)\,\text{d}t\dsx,\,\,u\in\mathrm{H}^1({B}_R),
\end{align*}
where $F$ is a primitive of $f$:
\begin{align}\label{bigF}
  F(t)=\int_0^tf(s)\,\text{d}s.  
\end{align}
Under assumptions with physical meanings that $f:\mathbb{R}\to\mathbb{R}$ is strictly increasing and $g:\mathbb{R}\to\mathbb{R}$ is monotonically decreasing and non-negative, the author in his previous work~\cite{Lee2020} applied the standard direct method to $\mathcal{E}$ and established the existence of weak solutions to \eqref{ndu}.  Then, following the standard argument consisting of the maximum principle and the elliptic regularity theorem (cf. \cite{GT1983}),  (\ref{ndu}) he obtained the uniqueness of solutions to \eqref{ndu}. Furthermore, by the uniqueness this solution is radially symmetric in $\overline{B_R}$, and satisfies \eqref{equ}--\eqref{bdu}.

 To the best of our knowledge, however, when $g$ is not necessary a decreasing function, the issue about the existence result of equation~\eqref{ndu} and its radial version~\eqref{equ}--\eqref{bdu} remains to be open. Based on \cite{Lee2020}, in this note we shall focus on the radial version~\eqref{equ}--\eqref{bdu} and assume that
\begin{align}\label{as-f}
f\in\text{C}(\mathbb{R};\mathbb{R})\,\,\text{is\,\,strictly\,\,increasing},\,\,f(0)=0\,\,\text{and}\,\,\liminf_{t\to0}\frac{f(t)}{t}>0.   
\end{align}

Moreover, by \eqref{as-f}, $F\in\text{C}^1(\mathbb{R};\mathbb{R})$ is strictly convex  and has the minimum value $F(0)=0$ in $\mathbb{R}$. We further make an assumption for $F$:  there exists $\theta_0>1$ such that
\begin{align}\label{f-F}
    tf(t)\geq\theta_0F(t)\quad\text{for}\quad\,|t|\gg1.
\end{align}
Note that, by \eqref{f-F}, there holds $\lim\limits_{|t|\to\infty}F(t)=\infty$. When $\theta_0>2$,
\eqref{f-F} particularly implies that $f$ and $F$ are superlinear at infinity. Such an assumption was introduced by Ambrosetti and Rabinowitz \cite{AR1973}. Besides, an application of \eqref{as-f}--\eqref{f-F} is $f(u)=\sinh{u}$ which appears in the so-called Poisson--Boltzmann equation~\cite{BS1984,L2016,S1993} and sinh--Gordon equation~\cite{JK1990}.

With these properties of $f$ and $F$, we propose a condition of $g$ for the non-existence of \eqref{equ}--\eqref{bdu} as $R>0$ is sufficiently large, which is stated as follows.
\begin{theorem}\label{m-thm}
Under \eqref{as-f}--\eqref{f-F}, if $g\in\mathrm{C}(\mathbb{R};\mathbb{R})$ satisfies 
\begin{align}\label{as-g}
 g^2(t)\neq2F(t),\,\,\forall\,t\in\mathbb{R},\,\,\text{and}\,\,\lim_{|t|\to\infty}\frac{g^2(t)}{2F(t)} \neq1,  
\end{align}
then there exists $R^*=R^*(f,g)>0$ depending on $f$ and $g$ such that when $R>R^*$, equation~\eqref{equ}--\eqref{bdu} has no solution.
\end{theorem}
An example of \eqref{as-g} is $f(t)=\sinh t$ and $g(t)=\pm\left(1+4\sinh\frac{|t|}{2}\right)$.  In Section~\ref{m-sec}, we will state the proof of Theorem~\ref{m-thm}.
\begin{remark}
It should be stressed that $g$ satisfying \eqref{as-g} is not a decreasing function. Note that by \eqref{bigF} and \eqref{as-g} we have $g(0)\neq0$.  Without loss of generality, we may assume $g(0)<0=F(0)$. Suppose on the contrary that $g$ is decreasing on $\mathbb{R}$. Then $g^2$ is increasing on $(t_-,0]$, where $t_-$ is finite such that $g(t_-)=0$ or $t_-=-\infty$ if $g<0$ on $(-\infty,0)$. Since $2F$ is  strictly decreasing on $(-\infty,0)$ with $\lim\limits_{t\to-\infty}F(t)=\infty$ and $g^2(0)>2F(0)$, by the immediate value theorem there exists $t_0\in(t_-,0)$ such that $g^2(t_0)=2F(t_0)$. This contradicts to \eqref{as-g}. 
\end{remark}
\begin{remark}
As an application of Theorem~\ref{m-thm}, we shall point out that if $g$ satisfies one of the following statements (i) and (ii), then as $R>0$ is sufficiently large, equation~\eqref{equ}--\eqref{bdu} has no solution:
\begin{itemize}
    \item[(i)] $g(0)>0$ and $\inf\limits_{t\in\mathbb{R}\setminus\{0\}}\frac{g(t)}{\sqrt{2F(t)}}>1$; 
        \item[(ii)] $g(0)<0$ and $\sup\limits_{t\in\mathbb{R}\setminus\{0\}}\frac{g(t)}{\sqrt{2F(t)}}<-1$.
\end{itemize}
\end{remark}

Theorem~\ref{m-thm} also shows that under \eqref{as-f}--\eqref{f-F}, if $g\in\mathrm{C}(\mathbb{R};\mathbb{R})$ satisfies \eqref{as-g}, then as $R>R^*$,  equation~\eqref{ndu} has no radially symmetric solution. Accordingly, we shall state a problem which, to the best of our knowledge, is unsolved.\vspace{6pt}\\
{$\blacksquare$~\bf Open problem.}  Assume that $f$ and $F$ satisfy \eqref{as-f}--\eqref{f-F} and $g$ satisfies \eqref{as-g}. Does there exist $R_*>0$ such that for each $R>R_*$, equation~\eqref{ndu} has a solution which is not radially symmetric in $\overline{B_R}$?

\section{\bf Proof of Theorem~\ref{m-thm}}\label{m-sec}
We first consider a change of variables
\begin{align}\label{ch-v}
U(x)=u(r)\,\,\text{with}\,\,  x=\frac{r}{R}\,\,\text{and}\,\,\eps=\frac{1}{R}>0.  
\end{align}
Then \eqref{equ}--\eqref{bdu}  is equivalent to the equation
\begin{align}
   -\eps^2&\left(U''(x)+\frac{N-1}{x}U'(x)\right)+f(U(x))=0,\quad\,x\in(0,1),\label{equ2}\\
  &U'(0)=0,\quad \eps{U'(1)}=g(U(1)).\label{bdu2}
\end{align}
Moreover, from \eqref{bdu2} let us set
\begin{align}\label{newU}
    U(1)=\lambda_{\eps},\quad\eps{U'(1)}=g(\lambda_{\eps}).
\end{align}
Although the solution $U$ depends on the parameter $\eps$ and should be denoted by $U_{\eps}$, without the confusion we omit its subscript for a sake of simplicity.

Equation~\eqref{equ2} with the boundary condition \eqref{newU} is not an overdetermined problem since $\lambda_{\eps}$ will be determined. Note also that assumption~\eqref{as-f} implies $f(U(x))=C(x)U(x)$ for some function $C(x)>0$. Thus, for each $\eps>0$ and $\lambda_{\eps}\in\mathbb{R}$, equation~\eqref{equ2} with the boundary conditions~$(U'(0),U(1))=(0,\lambda_{\eps})$ satisfies the maximum principle and has a unique classical solution~(see, e.g., \cite{C1996} and \cite[Section~2]{Lee2020}). To be specific we shall study the asymptotics (with respect to $\eps\downarrow0$) of solutions to equation~\eqref{equ2} with boundary conditions~$(U'(0),U(1))=(0,\lambda_{\eps})$. In doing so, it is expected to obtain the refined asymptotics of $U'(1)$ so that we can further investigate $\lambda_{\eps}$ via $\eps{U'(1)}=g(\lambda_{\eps})$ with respect to $\eps\downarrow0$.

 \begin{lemma}\label{lem1} Let $U$ be the unique classical solution of \eqref{equ2} with the boundary conditions~$(U'(0),U(1))=(0,\lambda_{\eps})$. Then, we have
\begin{align}\label{maxin}
    \min\{0,\lambda_{\eps}\}\leq\,U\leq \max\{0,\lambda_{\eps}\}\quad\text{and}\quad\lambda_{\eps}U'\geq0\quad\text{on}\,\,[0,1].
\end{align}
\end{lemma}
\begin{proof}[Proof of Lemma~\ref{lem1}]
We first assume $\lambda_{\eps}\geq0$. Suppose $U(0)<0$. Then there exists $\delta>0$ such that $U<0$ on $[0,\delta)$. Along with \eqref{equ2}, one may employ \eqref{as-f} to obtain $\eps^2(x^{N-1}U'(x))'=x^{N-1}f(U(x))<0$ on $(0,\delta)$. In particular, $U'<0$ on $(0,\delta)$. As a consequence, $U$ arrives at its minimum value at an interior point $x_0\in(0,1)$, and by \eqref{equ2} we get $f(U(x_0))\geq0$. This leads a contradiction since $U(x_0)\geq0>U(0)$. Hence, there holds $U(0)\geq0$. Applying the maximum principle to \eqref{equ2}, one arrives at $0\leq{U(0)}\leq{U(x)}\leq{U(1)=\lambda_{\eps}}$. Thus,
\begin{align}\label{ch-1st}
  \eps^2(x^{N-1}U'(x))'=x^{N-1}f(U(x))\geq0\quad\text{on}\,\,  (0,1),
\end{align}
and we further obtain $U'\geq0$ on $[0,1]$.

Similarly, for the case $\lambda_{\eps}<0$ we have  $\lambda_{\eps}={U(1)}\leq{U(x)}\leq{U(0)\leq0}$ and $U'\leq0$ on $[0,1]$. This completes the proof of \eqref{maxin}.
\end{proof}
When $\lambda_{\eps}=0$, \eqref{maxin} implies that equation \eqref{equ2} with the boundary condition ~$(U'(0),U(1))=(0,\lambda_{\eps})$ only has a trivial solution $U\equiv0$, together with \eqref{newU} we obtain $g(0)=0$. This is impossible  due to \eqref{as-g}.
In what follows, without loss of generality, {\bf it suffices to consider the case $\boldsymbol{\lambda_{\eps}>0}$}. Hence, we have $U\geq0$, $f(U)\geq0$ and $U'\geq0$ on $[0,1]$. Moreover, \eqref{ch-1st} holds, and we have the following estimates.
\begin{lemma}\label{lem2}
Under the same assumptions as in Lemma~\ref{lem1}, we assume $\lambda_{\eps}>0$. Then, there exists a positive constant $M$ independent of $\eps$ such that as $\eps\in(0,\frac{M}{\sqrt{2}(N-1)})$, we have
\begin{equation}\label{0ineq0}
    0\leq{U(x)}\leq2\lambda_{\eps}\exp\left(-\frac{M}{\eps}(1-x)\right),\quad\,x\in[0,1],
\end{equation}
and that: 
\begin{itemize}
    \item[(i)] If $\limsup\limits_{\eps\downarrow0}\lambda_{\eps}<\infty$, then
    \begin{align}\label{obds-1}
    \lim_{\eps\downarrow0}\left(\frac{g^2(\lambda_{\eps})}{2}-F(\lambda_{\eps})\right)=0.
    \end{align}
        \item[(ii)] If $\lambda_{\eps}\xrightarrow{\eps\downarrow0}\infty$, then
       \begin{equation}\label{obds-2}
    \lim_{\eps\downarrow0}\frac{g^2(\lambda_{\eps})}{2F(\lambda_{\eps})}=1.
\end{equation}
\end{itemize}
\end{lemma}
\begin{proof}[Proof of Lemma~\ref{lem2}]
Multiplying \eqref{equ2} by $U$ and using \eqref{as-f} and \eqref{f-F}, one may check that
\begin{align*}
    \frac{\eps^2}{2}(U^2(x))''=&\,\eps^2\left((U'(x))^2-\frac{N-1}{x}U(x)U'(x)\right)+U(x)f(U(x))\\
    \geq&\,\left(M^2-\eps^2\frac{(N-1)^2}{4x^2}\right)U^2(x).
\end{align*}
Here we have used \eqref{as-f} and \eqref{f-F} to verify a positive constant $M$ independent of $\eps$ such that $tf(t)\geq{M}t^2$ for all $t\in\mathbb{R}$. As a consequence, for $0<\eps^{\star}<\frac{\sqrt{2}M}{N-1}$, we have
\begin{align*}
    \eps^2(U^2(x))''
    \geq\,M^2U^2(x),\quad\,x\in[\textstyle\frac{N-1}{\sqrt{2}M}\eps^{\star},1)\,\,\text{and}\,\,\eps\in(0,\eps^{\star}).
\end{align*}
Along with \eqref{maxin} for $\lambda_{\eps}>0$, we follow the comparison theorem to obtain
\begin{align}\label{2ineq1}
    0\leq{U(x)}\leq\lambda_{\eps}\left(\exp\left(-\frac{M(x-{\frac{N-1}{\sqrt{2}M}\eps^{\star}})}{\eps}\right)+\exp\left(-\frac{M(1-x)}{\eps}\right)\right),
\end{align}
for $x\in[\textstyle\frac{N-1}{\sqrt{2}M}\eps^{\star},1)$ and $\eps\in(0,\eps^{\star})$. In particular, for $x\in[0,\frac12(\frac{N-1}{\sqrt{2}M}\eps^{\star}+1)]$ with $\eps^{\star}=\frac{M}{\sqrt{2}(N-1)}$, \eqref{maxin} and \eqref{2ineq1} imply
\begin{equation}\label{ineq2u}
\begin{aligned}
    0\leq U(x)\leq&\, U({\textstyle\frac12(\frac{N-1}{\sqrt{2}M}\eps^{\star}+1)})\\
    \leq&\,2\lambda_{\eps}\exp\left(-\frac{M}{2\eps}(1-\frac{N-1}{\sqrt{2}M}\eps^{\star})\right)=2\lambda_{\eps}\exp\left(-\frac{M}{4\eps}\right)\\
    \leq&\,2\lambda_{\eps}\exp\left(-\frac{M}{4\eps}(1-x)\right),\quad\textstyle\text{for}\,\,x\in[0,\frac34]\,\,\text{and}\,\,\eps\in(0,\frac{M}{\sqrt{2}(N-1)}).
\end{aligned}    
\end{equation}
On the other hand, by \eqref{2ineq1} with $\eps^{\star}=\frac{M}{\sqrt{2}(N-1)}$, we have
\begin{equation}\label{2ineq3}
  \begin{aligned}
    0\leq{U(x)}\leq&\,\lambda_{\eps}\left(\exp\left(-\frac{M(x-\frac12)}{\eps}\right)+\exp\left(-\frac{M(1-x)}{\eps}\right)\right)\\
    \leq&\,2\lambda_{\eps}\exp\left(-\frac{M}{\eps}(1-x)\right),\quad\textstyle\text{for}\,\,x\in[\frac34,1]\,\,\text{and}\,\,\eps\in(0,\frac{M}{\sqrt{2}(N-1)}).
\end{aligned}  
\end{equation}
Therefore, \eqref{0ineq0} follows from \eqref{ineq2u} and \eqref{2ineq3}.

It remains to prove \eqref{obds-1}--\eqref{obds-2}. Multiplying \eqref{ch-1st} by $x^{N-1}U'(x)$, one may check via simple calculations that
\begin{align}\label{diu}
    \left(\frac{\eps^2}{2}x^{2N-2}U'^2(x)-x^{2N-2}F(U(x))\right)'=-(2N-2)x^{2N-3}F(U(x))
\end{align}
 Integrating \eqref{diu} over the interval~$(0,1)$ and using \eqref{newU}, we have
\begin{align}\label{asc}
  \frac{g^2(\lambda_{\eps})}{2}-F(\lambda_{\eps}) =-(2N-2)\int_0^1 x^{2N-3}F(U(x))\text{d}x
\end{align}
and the following two cases for the estimate of $\int_0^1 x^{2N-3}F(U(x))\text{d}x$:\\
{\bf Case~1.} When $\limsup\limits_{\eps\downarrow0}\lambda_{\eps}<\infty$, we assume $0<\lambda_{\eps}\leq{L^*}$ as $0<\eps\ll1$, where $L^*>0$ is independent of $\eps$. Then, by \eqref{as-f} and \eqref{0ineq0},
\begin{align}\label{ch-ca1}
0\leq\int_0^1 x^{2N-3}F(U(x))\text{d}x\leq\,f(L^*)\int_0^1 U(x)\text{d}x
\leq\frac{2L^*f(L^*)}{M}\eps\xrightarrow{\eps\downarrow0}0.
\end{align}
{\bf Case~2.} When $\lambda_{\eps}\xrightarrow{\eps\downarrow0}\infty$,  we notice that, by \eqref{bigF} and \eqref{f-F}, $(\frac{F(t)}{t})'=\frac{tf(t)-F(t)}{t^2}\geq\frac{(\theta_0-1)F(t)}{t^2}>0$ for $t\gg1$. Hence,  \eqref{maxin} gives $\sup\limits_{[0,1]}\frac{F(U)}{U}=\frac{F(U(1))}{U(1)}=\frac{F(\lambda_{\eps})}{\lambda_{\eps}}$, and
\begin{equation}\label{ch-ca2}
  \begin{aligned}
0\leq\int_0^1 x^{2N-3}F(U(x))\text{d}x\leq&\,\left(\sup_{[0,1]}\frac{F(U)}{U}\right)\int_0^1 U(x)\text{d}x  \\
(\text{by\,\,\eqref{0ineq0}})\,\,\leq&\,\frac{F(\lambda_{\eps})}{\lambda_{\eps}}\times\frac{2\lambda_{\eps}\eps}{M}=\frac{2F(\lambda_{\eps})}{M}\eps.
\end{aligned}  
\end{equation}
As a consequence, by \eqref{asc} and \eqref{ch-ca1}, we prove \eqref{obds-1}; by \eqref{asc} and \eqref{ch-ca2}, we prove \eqref{obds-2}. Thus, the proof of Lemma~\ref{lem2} is completed.
\end{proof}

Having Lemma \ref{lem2} in hands, we state the proof of Theorem~\ref{m-thm} as follows.
\begin{proof}[Proof of Theorem~\ref{m-thm}]
Suppose on the contrary that there exists a strictly increasing  sequence $R_i\xrightarrow{i\to\infty}\infty$ such that for each equation~\eqref{equ}--\eqref{bdu} corresponding to $R=R_i$ has a classical solution $u_i$. Then by \eqref{ch-v} we set $\eps_i=\frac{1}{R_i}\xrightarrow{i\to\infty}0$ and  $\lambda_{\eps_i}=U_i(1)=u_i(R_i)$. Note that the sequence $\{\lambda_{\eps_i}\}_{i\in\mathbb{N}}$ contains infinitely many members of non-negative numbers
 or non-positive numbers. Hence, without loss of generality, we may assume  $\lambda_{\eps_i}>0$, $\forall\,i\in\mathbb{N}$. (As mentioned previously, if $\lambda_{\eps_i}=0$, then $U_i\equiv0$ on $(0,1)$ and $g(0)=0$ which is impossible!)

We now consider two situations for $\{\lambda_{\eps_i}\}_{i\in\mathbb{N}}$. If $\limsup\limits_{i\to\infty}\lambda_{\eps_i}=\lambda^*<\infty$, then there exists a subsequence $\{\lambda_{\eps_{n_i}}\}$ such that $\lim\limits_{n_i\to\infty}\lambda_{\eps_{n_i}}=\lambda^*$. Since both $g$ and $F$ are continuous on $\mathbb{R}$, by Lemma~\ref{lem2}(i) we obtain $g^2(\lambda^*)=2F(\lambda^*)$ which contradicts to \eqref{as-g}. On the other hand, if $\lim\limits_{i\to\infty}\lambda_{\eps_i}=\infty$, then by Lemma~\ref{lem2}(ii) we have $\lim\limits_{i\to\infty}\frac{g^2(\lambda_{\eps_i})}{2F(\lambda_{\eps_i})}=1$  which still contradicts to \eqref{as-g}. Therefore,  there exists $R^*=R^*(f,g)>0$ depending on $f$ and $g$ such that when $R>R^*$, equation~\eqref{equ}--\eqref{bdu} has no solution. We thus complete the proof of Theorem~\ref{m-thm}. 
\end{proof}

\subsection*{Acknowledgement}
  This work was partially supported by the MOST grant 110-2115-M-007 -003 -MY2 of Taiwan.


\begin{thebibliography}{99}
 \bibitem{AR1973}{\sc A. Ambrosetti, P. Rabinowitz}, {\em Dual variational methods in critical point theory and applications}, J. Funct. Anal. {\bf14} (1973) 349--381.
\bibitem{A1976}{\sc H. Amann}, {\rm Nonlinear elliptic equations with nonlinear boundary conditions}, New Developments in Diff. Eq., North-Holland, Math. Studies {\bf21} (1976) 43--63.

\bibitem{BS1984}{\sc C. Bandle,  R.P. Sperb, I. Stakgold}, {\em  Diffusion and reaction with monotone kinetics}, Nonlinear Anal. {\bf 8} (1984) 321--333. 

\bibitem{C1996}{\sc N. C\'{o}nsul}, {\em On equilibrium solutions of diffusion equations with nonlinear boundary conditions}, Z. angew. Math. Phys. {\bf 47} (1996) 194--209.

\bibitem{D2006} {\sc M. Dehghan
}: {\em A computational study of the one‐dimensional parabolic equation subject to nonclassical boundary specifications}, Numerical Methods for Partial Differential Equations {\bf 22} (2006)  220--257.
\bibitem{GT1983} {\sc D. Gilbarg, N.S. Trudinger}, 
{\em Elliptic partial differential equations of second order}, 
Springer-Verlag, New York, Heidelberg, and Berlin, 1983.
\bibitem{JK1990}{\sc M. Jaworski, D. Kaup}, {\em Direct and inverse scattering problem associated with the elliptic sinh--Gordon equation}, Inverse Problems {\bf 6} (1990) 543--556.

\bibitem{L2016}{\sc C.-C. Lee}, {\em  Effects of the bulk volume fraction on solutions of modified Poisson--Boltzmann equations}, J. Math. Anal. Appl. {\bf 437} (2016) 1101--1129.
\bibitem{Lee2020}{\sc C.-C. Lee}, {\em Nontrivial boundary structure in a Neumann problem on balls with radii tending to infinity}, Ann. Mat. Pura Appl. {\bf 199} (2020), no. 3, 1123--1146.

\bibitem{S1993}{\sc R. Sperb},  {\em Optimal bounds in semilinear elliptic problems with nonlinear boundary conditions}, Z. Angew. Math. Phys. {\bf 44} (1993) 639--653. 

























\end{thebibliography}
\end{document}